\documentclass[amstex,12pt,reqno]{amsart}
\usepackage{amsmath,amsthm,amscd,amsfonts,amssymb,enumitem}
\usepackage{latexsym}
\usepackage{multirow}
\usepackage{lipsum} 
\usepackage{enumitem}
\usepackage[colorlinks=true, linkcolor=blue, anchorcolor=blue, citecolor=blue, filecolor=blue, urlcolor=blue]{hyperref}

\newtheorem{thm}{Theorem}[section]
\newtheorem{lemma}[thm]{Lemma}

\theoremstyle{definition}
\newtheorem{defin}[thm]{Definition}

\theoremstyle{remark}

\newcommand{\qedwhite}{\hfill \ensuremath{\Box}}
\renewenvironment{proof}{{\raggedright \bfseries Proof.}}{\qedwhite}

\newcommand{\floor}[1]{\left\lfloor {#1} \right\rfloor}
\newcommand{\ceil}[1]{\left\lceil {#1} \right\rceil}
\newcommand\numberthis{\addtocounter{equation}{1}\tag{\theequation}}
\numberwithin{equation}{section}

\def\Chi{\mbox{\large$\chi$}}
\begin{document}

\leftline{ \scriptsize \it  }
\title[]{Weighted approximation By Max-product Generalized Exponential Sampling Series}
\maketitle

\begin{center}
    \bf  $^1${Satyaranjan Pradhan},
    \bf  $^2${Madan Mohan Soren}
    \vskip0.2in
    $^{1,2}${Department of Mathematics, Berhampur University,Odisha-760007} \\
    \verb"satyaranjanpradhan6312@gmail.com",
    \verb"mms.math@buodisha.edu.in "
\end{center}

\begin{abstract}
In this article, we study the convergence behaviour of the classical generalized Max Product exponential sampling series in the weighted space of log-uniformly continuous and bounded functions. We derive basic convergence results for both the series and study the asymptotic convergence behaviour. Some quantitative approximation results have been obtained utilizing the notion of weighted logarithmic modulus of continuity.
 
\end{abstract}


\section{Introduction} 
The problem of sampling and reconstruction of functions is a fundamental aspect of approximation theory, with important applications in signal analysis and image processing (\cite{apl1,apl2}). A significant breakthrough in sampling and reconstruction theory was collectively achieved by Whittaker-Kotelnikov-Shannon. They established that any band-limited signal $f$, i.e. the Fourier transform of $f$ is compactly supported, can be completely recovered using its regularly spaced sample values (see \cite{srvbut}). This result is widely known as \textit{WKS sampling theorem}. Butzer and Stens \cite{but1} generalized this result significantly for not-necessarily band-limited signals. Since then, several mathematicians have been making significant advancements in this direction, see \cite{butzer2,k2007,tam}. 
\par

The problem of approximating functions with their exponentially-spaced sample values can be traced back to the work of Ostrowski et.al. \cite{ostrowsky}, Bertero and Pike \cite{bertero}, and Gori \cite{gori}. In order to deal with exponentially-spaced data, they provided a series representation for the class of Mellin band-limited functions (defined in Section \ref{section2}). This reconstruction formula is referred as the \textit{exponential sampling formula} and defined as follows. For $f:\mathbb{R}^{+} \rightarrow \mathbb{C}$ and $c \in \mathbb{R},$ the exponential sampling formula is given by (see \cite{butzer3})

\begin{equation} \label{expformula}
(E_{c,T}f)(x):= \sum_{k=-\infty}^{\infty} lin_{\frac{c}{T}}(e^{-k}x^{T}) f(e^{\frac{k}{T}})
\end{equation}
where $lin_{c}(x)= \dfrac{x^{-c}}{2\pi i} \dfrac{x^{\pi i}-x^{-\pi i}}{\log c} = x^{-c} sinc(\log x)$ with continuous extension $lin_{c}(1)=1.$ Moreover, if $f$ is Mellin band-limited to $[-T,T],$ then $(E_{c,T}f)(x)=f(x)$ for each $x \in \mathbb{R}^{+}.$
\par 

The exponentially spaced data can be observed in various problems emerging in optical physics and engineering, for example, Fraunhofer diffraction, polydispersity analysis by photon correlation spectroscopy, neuron scattering, radio astronomy, etc (see \cite{casasent,ostrowsky,bertero,gori}). Therefore, it became crucial to examine the extensions and variations of the exponential sampling formula \eqref{expformula}.Butzer and Jansche \cite{butzer5} investigated into the exponential sampling formula, incorporating the analytical tools of Mellin analysis. They established that the theory of Mellin transform provides a suitable framework to handle sampling and approximation problem related to exponentially-spaced data. The foundational work on the Mellin transform theory was initially undertaken by Mamedov \cite{mamedeo}. Subsequently, Butzer and his colleagues made significant contributions to the field of Mellin theory in \cite{butzer3,butzer5}. For some notable developments on Mellin theory, we refer to \cite{bardaro1,bardaro9,bardaro2,bardaro3} etc. In order to approximate a function which is not necessarily Mellin band-limited, the theory of exponential sampling formula \eqref{expformula} was extended in \cite{bardaro7} using generalized kernel satisfying suitable conditions. This gives a method to approximate the class of log-continuous functions by employing its exponentially spaced sample values. For $x \in \mathbb{R}^{+}$ and $w>0,$ the generalized exponential sampling series is given by (see \cite{bardaro7})

\begin{equation} \label{genexp}
(S_{w}^{\chi}f)(x)= \sum_{k=- \infty}^{\infty} \chi(e^{-k} x^{w}) f( e^{\frac{k}{w}})
\end{equation}

for any $ f: \mathbb{R}^{+} \rightarrow \mathbb{R}$ such that the series \eqref{genexp} converges absolutely. Various approximation properties associated with the family of operators \eqref{genexp} can be observed in \cite{comboexp,bardaro11,bevi,diskant}. The approximation properties of exponential sampling operators based on artificial neural network can be found in \cite{sn,self}. In order to approximate integrable functions, the series \eqref{genexp} is not suitable. To overcome with this, the following Kantorovich type modification of the family \eqref{genexp} was studied in \cite{own}. For $ x\in \mathbb{R}^{+}, k \in \mathbb{Z}$ and $w>0,$ the Kantorovich  exponential sampling series is defined by

\begin{equation} \label{kant}
(I_{w}^{\chi}f)(x):= \sum_{k= - \infty}^{\infty} \chi(e^{-k} x^{w})\  w \int_{\frac{k}{w}}^{\frac{k+1}{w}} f(e^{u})\  du \ \
\end{equation}

whenever the series (\ref{kant}) is absolutely convergent for any locally integrable function $ f: \mathbb{R}^{+} \rightarrow \mathbb{R}.$\\
Recently, an interesting procedure was introduced by Coroianu and Gal, which allows to improve the order of approximation that can be achieved by a family of linear operators, and consists in the
so-called max-product approach, see e.g. Coroianu and Gal (2010)\cite{CG2010}, Coroianu and Gal (2011)\cite{CG2011} and Coroianu and Gal (2012)\cite{CG2012}. In general, discrete linear operators are defined by finite sums or series, with respect to certain indexes. The max-product operators are defined by replacing the sums or the series by a maximum or a supremum,computed over the same sets of indexes.The above procedure, allows to convert linear operators into nonlinear ones, which are able to achieve a higher order of approximation with respect to their linear counterparts; see, e.g.,Coroianu and Gal (2010)\cite{CG2010}, Coroianu and Gal (2011)\cite{CG2011}, and Coroianu and Gal (2012)\cite{CG2012}.
\par

In the present paper, we investigate the approximation behaviors of  Max Product generalized exponential sampling operators for functions belonging to weighted space of functions. More precisely, for the above operators the point-wise and uniform convergence are established together with the rate of convergence by involving the weighted modulus of continuity. Finally, we prove quantitative Voronovskaja type theorem for the operators
\par
The paper is organized as follows. The first and second sections are devoted to fundamentals of Max product generalized exponential sampling , as well as to some auxiliary results. The third section is devoted to well-definiteness of the operators between weighted spaces of functions, together with pointwise and uniform convergence. In Section 4, we present rate of convergence of the family of operators in terms of weighted modulus of continuity. The last section contains a pointwise convergence result in quantitative form by means of a Voronovskaja type theorem.

\section{Preliminaries}\label{section2} 

In this paper, let $\mathrm{I} = [a,b]$ be any compact subset of $\mathrm{R}_{+}$, Where $\mathrm{R}_{+}$ denotes the set of positive real numbers.We denote $\mathrm{C(I)}$ the space of uniformly continuous function on I, with usual supremum norm  \[\|f\|_{\infty} := \sup_{x\in I} |f(x)|\]
and $\mathrm{C_{+}(I)}$, the subspace of $\mathrm{C(I)}$ which contains non negative function defined on $\mathrm{I}$.\\
We now define log uniformly continuous functions.we say that a function $f:\mathrm{R}_{+} \rightarrow \mathrm{R}$ is log uniformly continuous function if for given $\epsilon > 0$ there exist $\delta >0$ such that $ |f(x) - f(y)| < \epsilon $ whenever $ |\log(x) -\log(y)| \leq \delta$ where $x,y \in \mathrm{R}_{+} $
\par
We denote $\mathcal{U}(\mathrm{I})$, the space of all log uniformly continuous function defined on $\mathrm{I}$ and  $\mathcal{U_{+}}(\mathrm{I})$, the space of all non negative log uniformly continuous functions on $\mathrm{I}$.we further denote $\mathrm{B_{+}}(\mathrm{I}) $, the space of all bounded and non-negative function on $\mathrm{I}$.
Now we define the notation which will be useful to analyze the max product sampling operators.\\
For any index set $ \Lambda \subseteq \mathrm{Z}$ we have \[ \bigvee_{\mathrm{l}\in \Lambda } \mathrm{T_{l}} := \sup\{ \mathrm{T_{l}}  :  \mathrm{l} \in \Lambda\} \]
If $\Lambda $ is finite , then \[ \bigvee_{\mathrm{l}\in \Lambda } \mathrm{T_{l}} = \max_{\mathrm{l}\in \Lambda } \mathrm{T_{l}} \]
\par
Now we introduce the functions that will be used as kernels of the max product generalized exponential sampling operator.\\
in what follows. We define a kernel any bounded and measurable function $\Chi :\mathrm{R}_{+} \rightarrow \mathrm{R} $ which satisfies following conditions.
\begin{enumerate}
    \item[$(\chi_{1}) :$] There exists $\mu >0$ ,such that the discrete absolute moment of order $\mu $  \[m_{\mu}(\Chi)= \sup_{u\in \mathrm{R_{+}}} \bigvee_{k\in z} |\Chi(e^{-k}u)||k-\log u|^{\mu} \;   \text{is finite for $\mu=2$}\].
    \item[$(\chi_{2}):$] There exist \[\inf_{x \in [1,e]} \Chi(x) =: \eta(x)\]
\end{enumerate}
	
\begin{lemma}
    Let $\Chi$ be a bounded function satisfying $(\Chi_{1})$ with $\mu > 0$.Then $m_{\nu}(\Chi) < \infty $ for every $ 0 \leq \nu \leq \mu $   
\end{lemma}
\begin{proof}
    Let $ 0 \leq \nu \leq \mu $ be fixed.For every $u \in \mathrm{R_{+}}$ we have
    \begin{align*}
        \bigvee_{k\in z} |\Chi(e^{-k}u)||k-\log u|^{\nu} \leq & \bigvee_{k\in \mathbb{Z},|k-\log(u)| \leq 1} |\Chi(e^{-k}u)||k-\log u|^{\nu} \\ & + \bigvee_{k\in \mathbb{Z}, |k-\log(u)| > 1} |\Chi(e^{-k}u)||k-\log u|^{\nu}\\
        \leq &  m_{0}(\Chi) +  m_{\nu}(\Chi)\\
        \leq & \infty
    \end{align*}
\end{proof}
\begin{lemma}
    Let $\Chi : \mathrm{R_{+}}\rightarrow \mathrm{R}$ be a kernel satisfying that $ m_{\nu}(\Chi) < \infty, \nu > 0$. Then for every $\delta >0$, we have \[ \bigvee_{k\in \mathbb{Z},|k-\log(u)| > \delta \mathrm{w} } |\Chi(e^{-k}u)| = \mathcal{O}(\mathrm{w^{-\nu}}) \; \text{as $ \mathrm{w} \rightarrow \infty $}\] uniformly with respect to $\mathrm{u} \in \mathrm{R_{+}}$.
\end{lemma}
\begin{proof}
    Let $\mathrm{u} \in \mathrm{R_{+}}$ be fixed .Then we have
    \begin{align*}
        \bigvee_{k\in \mathbb{Z},|k-\log(u)| > \delta \mathrm{w} } |\Chi(e^{-k}u)| = &\bigvee_{k\in \mathbb{Z},|k-\log(u)| > \delta \mathrm{w} } |\Chi(e^{-k}u)| \frac{|k-\log(u)|^{\nu}}{|k-\log(u)|^{\nu}}\\
        \leq & \frac{1}{(\delta \mathrm{w})^{\nu}}\bigvee_{k\in \mathbb{Z},|k-\log(u)| > \delta \mathrm{w} } |\Chi(e^{-k}u)| |k-\log(u)|^{\nu}\\
        \leq & \frac{m_{\nu}(\Chi)}{(\delta \mathrm{w})^{\nu}} < +\infty
    \end{align*} 
    Hence proof is completed
\end{proof}
\begin{lemma}
    Let $\Chi : \mathrm{R_{+}}\rightarrow \mathrm{R}$ be a kernel satisfying the condition $(\Chi_{2})$ .Then we have the following
    \begin{enumerate}
        \item For any $x \in \mathrm{I}$, we have  \[\bigvee_{k\in \mathbb{J}_{w} } |\Chi(e^{-k}x^{\mathrm{w}})| \geq \eta(x) \] where $\mathbb{J}_{w} = \{ k \in \mathbb{Z} : k = \ceil{w\log a}.....,\floor{w\log b} \}$
        \item For any $x \in \mathrm{R_{+}}$, we have  \[\bigvee_{k\in \mathbb{Z}} |\Chi(e^{-k}x^{\mathrm{w}})| \geq \eta(x)\]
    \end{enumerate}  
\end{lemma}
\begin{proof}
    We prove the first part of the lemma. Since the second part can be deduced analogously.we first observe the following. Suppose that $x\in [a,b]$ be fixed. then there exist at least a $ k_{1} \in \mathbb{J}_{w}$ such that $e^{-k}x^{\mathrm{w}} \in [1,e] $.\\
    Thus we obtain \[ \bigvee_{k\in \mathbb{Z}} \Chi(e^{-k}x^{\mathrm{w}}) \geq \Chi(e^{-k_{1}}x^{\mathrm{w}}) \geq \eta(x)\]
    Hence, proof is completed.
\end{proof}
 
For a particular weight function as $\mathrm{\overline{\omega}}( x) = \frac{1}{1+\log^{2}x}$.Let us assume that
\begin{itemize}[label={}]
    \item $\mathcal{B}(\mathrm{R}_{+}) := $ The space of all bounded functions on $\mathrm{R_{+}}$
    \item $\mathcal{UB}(\mathrm{R_{+}}):= $ The space of all log uniformly continuous and bounded \\function  defined on $\mathrm{R_{+}}$
    \item $\mathcal{UB}_{+}(\mathrm{R_{+}}):=$ The space of all non negative  log uniformly continuous\\ and bounded function defined on $\mathrm{R_{+}}$
\end{itemize}
Subsequently, we denote the weighted function space associated with $\mathcal{B}(\mathrm{R}_{+})$ as \\$\mathcal{B}^{\overline{\omega}}(\mathrm{R}_{+}) :=\{ f:\mathrm{R_{+} }\to \mathrm{R} :  \exists M > 0$ such that $\mathrm{\overline{\omega}}(x)|f(x)|\leq M $  $\forall x\in \mathrm{R_{+}}\}$. Simillarly we can define the weighted space associated with $\mathcal{UB}(\mathrm{R_{+}})$ as $\mathcal{UB}^{\overline{\omega}}(\mathrm{R}_{+}) := \{f:\mathrm{R_{+} }\to \mathrm{R} : \overline{\omega} f  \in \mathcal{UB}(\mathrm{R_{+}}) \}$ and $\mathcal{UB}^{\overline{\omega}}_{+}(\mathrm{R_{+}})$, the subclass of $\mathcal{UB}^{\overline{\omega}}(\mathrm{R}_{+})$ which is non negative.

It is interesting to mention that the linear space of functions  $\mathcal{UB}^{\overline{\omega}}(\mathrm{R}_{+})$  is normed linear space with the norm \[ \| f\|_{\overline{\omega}} = \sup_{x>0} \mathbf{\overline{\omega}}(x)|f(x)|\]

\begin{defin}
    The weighted logarithmic modulus of continuity for $f \in \mathcal{C}_2(\mathbf{R}^{+})$ and $\delta > 0$ is considered as 
    \[ \Omega(f,\delta)= \sup_{|\log t| \leq \delta , x > 0} \frac{|f(tx)-f(x)|}{(1+\log^{2}x) (1+\log^{2}t)}\]   
\end{defin}

\begin{lemma}
    The weighted logarithmic modulus of continuity $\Omega(f,\delta)$ has the following fundamental Properties:
    \begin{enumerate}
	\item $\Omega(f,\delta)$ is a monotonically increasing function of $\delta$
	\item  For $f \in \mathcal{C}_2(\mathbf{R}^{+})$ the quantity $\Omega(f,\delta)$ is finite.
	\item For all $f \in \mathcal{C}_2(\mathbf{R}^{+})$  and each $ \lambda \in \mathbf{R}^{+}$
	\[\Omega(f,\lambda \delta)\leq 2 (1+\lambda)^{3} (1+\delta^{2})\mu (f,\delta)\]
	\item For all $f \in \mathcal{C}_2(\mathbf{R}^{+})$ and $h,x >0$
	\[ |(f(h)- f(x)|\leq 16 (1+\delta^{2})^{2} (1+\log^{2}x) (1+ \frac{|\log h-\log x|^{5}}{\delta^{5}}) \Omega(f,\delta)\]
	\item For $f \in \mathcal{C}^{*}_2(\mathbf{R}^{+}),$ we have $\lim_{\delta \to 0} \Omega(f,\delta) = 0.$
    \end{enumerate}  
\end{lemma}

Let $ f : \mathrm{I} \to \mathrm{R} $ be any locally integrable function on $\mathrm{I}$.Let $\Chi$ be a kernel function satisfying $\bigvee_{\substack{k\in \mathbb{J}_{w}}} \Chi(e^{-k}x^{\mathrm{w}}) \neq 0 , \forall  x \in \mathrm{I}$\\
The max product exponential sampling operators are defined by \[ MG^{\chi}_{\mathrm{w}}(f,x) := \frac{\bigvee_{k\in \mathbb{J}_{w} } \Chi(e^{-k}x^{\mathrm{w}}) f(e^{\frac{k}{w}})}{\bigvee_{k\in \mathbb{J}_{w} } \Chi(e^{-k}x^{\mathrm{w}})}\] where $\mathbb{J}_{w} = \{ k \in \mathbb{Z} : k = \ceil{w\log a}.....,\floor{w\log b} \}$
\begin{lemma}
    Let $f,g \in \mathcal{B_{+}}(\mathrm{I}) $ be any locally integrable functions on I . Then we have
    \begin{enumerate}
        \item If  $ f(x) \leq g(x)  then MG^{\chi}_{\mathrm{w}}(f,x) \leq MG^{\chi}_{\mathrm{w}}(g,x) $
        \item $MG^{\chi}_{\mathrm{w}}(f+g,x) \leq MG^{\chi}_{\mathrm{w}}(f,x)+ MG^{\chi}_{\mathrm{w}}(g,x)$
        \item $|MG^{\Chi}_{\mathrm{w}}(f,x)- MG^{\chi}_{\mathrm{w}}(g,x)| \leq MG^{\chi}_{\mathrm{w}}(|f-g|,x) $ for all $x\in \mathrm{I}$
        \item $MG^{\chi}_{\mathrm{w}}(\lambda f,x) = \lambda \hspace{0.1in} MG^{\chi}_{\mathrm{w}}(f,x)$ for every $\lambda > 0$
    \end{enumerate}
\end{lemma}
\begin{proof}
    we can easily prove (1) (2) and (4) by using definitions of operators $MG^{\chi}_{\mathrm{w}}$.Now we prove third part of the result.we see that $ f(x) \leq |f(x) - g(x)| + g(x) $ and $ g(x) \leq |g(x) - f(x)| + f(x) $. By using the properties of (1) and (2) , we get  \[ MG^{\chi}_{\mathrm{w}}(f,x) \leq MG^{\chi}_{\mathrm{w}}(|f-g|,x)+ MG^{\chi}_{\mathrm{w}}(g,x)\] and \[ MG^{\chi}_{\mathrm{w}}(g,x) \leq MG^{\chi}_{\mathrm{w}}(|g-f|,x)+ MG^{\chi}_{\mathrm{w}}(f,x)\]
    Combining the above inequalities , we obtain \[|MG^{\Chi}_{\mathrm{w}}(f,x)- MG^{\chi}_{\mathrm{w}}(g,x)| \leq MG^{\chi}_{\mathrm{w}}(|f-g|,x) \text{\hspace{0.1in} $ \forall x\in \mathrm{I}$}\]
\end{proof}

\section{\textbf{Main Results for Generalized Exponential Sampling Series}} \label{section3}

The first main result shows that the operators $S^{\chi}_{w}$ are well defined on logarithmic weighted space of functions. we need the following preliminary propositions.

\begin{thm}
    Let $\chi : \mathrm{R_{+}} \to \mathrm{R}$ be the kernel function satisfying $(\Chi_{1})$ and $\Chi_{2}$ and we denote $\Psi(x) = \frac{1}{\overline{\mathrm{w}}} = 1+ \log^{2}(x)$  for all $x \in \mathrm{R_{+}}$. Then 
    \[ |MG^{\chi}_{\mathrm{w}}(\Psi,x)| \leq \frac{(1+\log^{2}x)}{\eta(x)}[m_0(\chi) + \frac{2}{w}m_{1}(\chi) + \frac{1}{w^{2}}m_{2}(\chi)]\]
\end{thm}
    
\begin{proof}
    For $ w > 0 , x \in \mathbf{R}^{+}$ and $k \in \mathbf{Z}$.
    \begin{equation*}
	\begin{split}
		\psi(e^{\frac{k}{w}}) & = 1+(\log(e^{\frac{k}{w}}))^{2}\\
		& = 1+ \left(\frac{k}{w}\right)^{2}\\
		& = 1+ \left(\frac{k}{w} - \log x\right)^{2} + 2\log x \left(\frac{k}{w}-\log x\right)+\log^{2} x
	\end{split}
    \end{equation*}
		
    Using the definition of the operator $MG^{\chi}_{\mathrm{w}}$ and in view of lemma 2.3, we have 
    \begin{equation*}
        \begin{split}
            |MG^{\chi}_{\mathrm{w}}(\psi,x)|  \leq  & \frac{1}{\eta(x)}\bigvee_{k\in \mathbf{Z}} \psi(e^{\frac{k}{w}}) |\chi(e^{-k}x^{w})|\\
		   = & \frac{1}{\eta(x)}\bigvee_{k\in \mathbf{Z}} \left( 1+ \left(\frac{k}{w} - \log x\right)^{2} + 2 \log x \left(\frac{k}{w}-\log x \right)+\log^{2} x \right) |\chi(e^{-k}x^{w})|\\ 
		   \leq & \frac{1}{\eta(x)} \bigvee_{k\in \mathbf{Z}} (1+\log^{2} x ) |\chi(e^{-k}x^{w})| + \bigvee_{k\in \mathbf{Z}} \left(\frac{k}{w} - \log x \right)^{2} |\chi(e^{-k}x^{w})| \\& +  2\bigvee_{k\in \mathbf{Z}} \log x \left(\frac{k}{w}-\log x \right) |\chi(e^{-k}x^{w})| \\
		   \leq & \frac{(1+\log^{2} x)}{\eta(x)}  \bigvee_{k\in \mathbf{Z}}|\chi(e^{-k}x^{w})| + \frac{1}{w^{2}} \bigvee_{k\in \mathbf{Z}} (k - w \log x)^{2} |\chi(e^{-k}x^{w})|\\ & +\frac{2|\log x|}{w}\bigvee_{k\in \mathbf{Z}} (k-w \log x) |\chi(e^{-k}x^{w})| \\ 
		   \leq & \frac{(1+\log^{2} x)}{\eta(x)}  \times \left( \bigvee_{k\in \mathbf{Z}}|\Chi(e^{-k}x^{w})| + \frac{1}{w^{2}} \bigvee_{k\in \mathbf{Z}} |(k - w\log x)|^{2} |\chi(e^{-k}x^{w})| \right)\\ & +\left(\frac{2}{w}\bigvee_{k\in \mathbf{Z}} |(k-w \log x)| |\chi(e^{-k}x^{w})| \right)\\
		   \leq & \frac{(1+\log^{2} x)}{\eta(x)}  \left[ m_{0}(\chi)+\frac{1}{w^2} m_{2}(\chi)+ \frac{2}{w}m_{1}(\chi)\right] 
        \end{split}
    \end{equation*}
which is desired.
\end{proof}

\begin{thm}
    Let $\chi$ be a kernel satisfying the assumptions $(\Chi_{1}), (\Chi_{2})$ for $\mu= 2$. Then for a fixed $\mathrm{w} >0$ the operator $MG^{\chi}_{\mathrm{w}}$ is a  operator from $\mathcal{B}_{\overline{\omega}}(R^{+}) \to \mathcal{B}_{\overline{\omega}}(R^{+})$ and its operator norm turns out to be
    \[ \| MG^{\chi}_{\mathrm{w}}\|_{\mathcal{B}_{\overline{\omega}}(R^{+}) \to \mathcal{B}_{\overline{\omega}}(R^{+})} \leq  \frac{1}{\eta^{2}(x)} \left[m_{0}(\chi)+\frac{1}{w^2} m_{2}(\chi)+ \frac{2}{w}m_{1}(\chi)\right]\]
\end{thm}
\begin{proof}
    Let us fix $w > 0$.Using the definition of new form $MG^{\chi}_{\mathrm{w}}(f,x)$ we have 
   \[ |MG^{\chi}_{\mathrm{w}}(f,x)| \leq \frac{1}{\eta(x)}\bigvee_{k\in \mathbf{Z}} \frac{ 1}{\overline{\omega}(e^{\frac{k}{w}})}  | \Chi(e^{-k}x^{\mathrm{w}})|   | \overline{\omega}(e^{\frac{k}{w}})f(e^{\frac{k}{w}})| \] 
    
    Moreover , since $f \in  \mathcal{B}_{\overline{\omega}}(R^{+}) $ and recalling $\Psi = \frac{1}{\overline{\omega}}$
		
    \begin{align*}
	|MG^{\chi}_{\mathrm{w}}(f,x)|  \leq  & \frac{\| f\|_{\overline{\omega}}}{\eta(x)} \bigvee_{k \in \mathbf{Z}} \Psi(e^{\frac{k}{w}}) |\chi(e^{-k}x^{w})|\\
	= & \frac{\| f\|_{\overline{\omega}}}{\eta(x)}  |MG^{\chi}_{\mathrm{w}}(\Psi,x)|\\
	\leq & \frac{\| f\|_{\overline{\omega}}}{\eta^{2}(x)} (1+\log^{2}x) \left[m_{0}(\chi)+\frac{1}{w^2} m_{2}(\chi)+ \frac{2}{w}m_{1}(\chi)\right]
    \end{align*}
    
    which implies that 
    \[\frac{|MG^{\chi}_{\mathrm{w}}(f,x)|}{(1+\log^{2}x)} \leq   \frac{\| f\|_{\overline{\omega}}}{\eta^{2}(x)}  \left[m_{0}(\chi)+\frac{1}{w^2} m_{2}(\chi)+ \frac{2}{w}m_{1}(\chi)\right]\; \text{for every $x\in \mathbf{R}^{+}$} \]
    Since the assumption $m_2(\chi) < +\infty$ implies $m_{j}(\chi) < + \infty $ for $j=0,1$\\
    We deduce $\| MG^{\chi}_{\mathrm{w}}\|_{w} < + \infty,$ i.e. $MG^{\chi}_{\mathrm{w}} \in \mathcal{B}_{\overline{\omega}}(R^{+}).$ On the other-hand taking supremum over $x\in \mathbf{R}^{+}$ in 
    \[\frac{|MG^{\chi}_{\mathrm{w}}(f,x)|}{(1+\log^{2}x)} \leq   \frac{\| f\|_{\overline{\omega}}}{\eta^{2}(x)}  \left[m_{0}(\chi)+\frac{1}{w^2} m_{2}(\chi)+ \frac{2}{w}m_{1}(\chi)\right]\] 
    and supremum with respect to $f \in \mathcal{B}_{\overline{\omega}}(R^{+})$ with $\| f\|_{\overline{\omega} }\leq 1 $ ,we have
    \[ \| MG^{\chi}_{\mathrm{w}}\|_{\mathcal{B}_{\overline{\omega}}(R^{+}) \to \mathcal{B}_{\overline{\omega}}(R^{+})} \leq  \frac{1}{\eta^{2}(x)} \left[m_{0}(\chi)+\frac{1}{w^2} m_{2}(\chi)+ \frac{2}{w}m_{1}(\chi)\right]\]
		
\end{proof}
	
\begin{thm}
Let $\Chi$ be a kernel satisfying $\Chi_{1} \And{ \Chi_{2}}$ for $\mu=2$. Let $f:\mathrm{R_{+}} \to\mathrm{R_{+}}$ be a bounded function.Then we have \[ \lim_{\mathrm{w} \to \infty } MG^{\chi}_{\mathrm{w}}(f,x) = f(x) \; \text{holds at each log-continuity point $x \in \mathrm{R_{+}}$} \]
Moreover if $f \in \mathcal{UB}^{\overline{\omega}}_{+}(\mathrm{R_{+}}) $, then \[\lim_{\mathrm{w} \to \infty } \|MG^{\chi}_{\mathrm{w}}(f,x) - f(x)\|_{\overline{\omega}} = 0 \]
\end{thm}

\begin{proof}
By the definition of the Max-Product generalized exponential sampling operator, we write 
\begin{align*}
    |MG^{\chi}_{\mathrm{w}}(f,x) - f(x)| \leq  & |MG^{\chi}_{\mathrm{w}}(f,x) - f(x) MG^{\chi}_{\mathrm{w}}(1,x)| + |f(x) MG^{\chi}_{\mathrm{w}}(1,x) - f(x)|\\
    \leq & |MG^{\chi}_{\mathrm{w}}(f,x) - f(x) MG^{\chi}_{\mathrm{w}}(1,x)| + |f(x) (MG^{\chi}_{\mathrm{w}}(1,x) - 1)| \\
\end{align*}
We now define $1:\mathrm{R_{+}} \to \mathrm{R_{+}}$ by $1(x)= x$ \text{and} $f_{x}: \mathrm{R_{+}} \to \mathrm{R_{+}}$ by $f_{x}(t) = f(x) $ for all $t\in \mathrm{R_{+}}$\\
Since $MG^{\chi}_{\mathrm{w}}(1,x)=1$, we have \[ |MG^{\chi}_{\mathrm{w}}(f,x) - f(x)| \leq MG^{\chi}_{\mathrm{w}}(|f-f_{x}|,x)\]
Using the definition of operator and view of lemma-2.3, we have
\[ MG^{\chi}_{\mathrm{w}}(|f-f_{x}|,x) \leq \frac{1}{\eta(x)} \bigvee_{k\in\mathbf{Z}} | \chi(e^{-k}x^{\mathrm{w}})| |f(e^{\frac{k}{\mathrm{w}}}- f(x)| \]
 Now for all $x \in  \mathrm{R_{+}} , k \in \mathbf{Z} \And{\mathrm{w}} > 0 $ 
 \begin{align*}
     f(e^{\frac{k}{\mathrm{w}}}) - f(x) = &  f(e^{\frac{k}{\mathrm{w}}} - \frac{\overline{\omega}(e^{\frac{k}{\mathrm{w}}}) f(e^{\frac{k}{\mathrm{w}}})}{\overline{\omega}(e^{\frac{k}{\mathrm{w}}})}  + \frac{\overline{\omega}(e^{\frac{k}{\mathrm{w}}}) f(e^{\frac{k}{\mathrm{w}}})}{\overline{\omega}(e^{\frac{k}{\mathrm{w}}})}  - f(x)\\
     = & \overline{\omega}(e^{\frac{k}{\mathrm{w}}}) f(e^{\frac{k}{\mathrm{w}}}) \left( \frac{1}{\overline{\omega}(e^{\frac{k}{\mathrm{w}}})} - \frac{1}{\overline{\omega}(x)}\right) + \frac{1}{\overline{\omega}(x)} \left( \overline{\omega}(e^{\frac{k}{\mathrm{w}}}) f(e^{\frac{k}{\mathrm{w}}}) -\overline{\omega}(x) f(x)\right )\\
     |f(e^{\frac{k}{\mathrm{w}}}) - f(x) | \leq &   \overline{\omega}(e^{\frac{k}{\mathrm{w}}}) |f(e^{\frac{k}{\mathrm{w}}})|  \left( \left| \frac{1}{\overline{\omega}(e^{\frac{k}{\mathrm{w}}})} - \frac{1}{\overline{\omega}(x)} \right|\right) + \frac{1}{\overline{\omega}(x)} \left( | \overline{\omega}(e^{\frac{k}{\mathrm{w}}}) f(e^{\frac{k}{\mathrm{w}}}) -\overline{\omega}(x) f(x)| \right)
 \end{align*}
 Thus we have 
\begin{align*}
    MG^{\chi}_{\mathrm{w}}(|f-f_{x}|,x) \leq & \frac{1}{\eta(x)} \bigvee_{k\in\mathbf{Z}} | \chi(e^{-k}x^{\mathrm{w}})| |f(e^{\frac{k}{\mathrm{w}}}- f(x)| \\
    \leq & \frac{1}{\eta(x)}  \bigvee_{k\in\mathbf{Z}} | \chi(e^{-k}x^{\mathrm{w}})| \left[ \overline{\omega}(e^{\frac{k}{\mathrm{w}}}) |f(e^{\frac{k}{\mathrm{w}}})|  \left( \left| \frac{1}{\overline{\omega}(e^{\frac{k}{\mathrm{w}}})} - \frac{1}{\overline{\omega}(x)} \right|\right) \right] \\& + \frac{1}{\eta(x)}\bigvee_{k\in\mathbf{Z}} | \chi(e^{-k}x^{\mathrm{w}})| \left[\frac{1}{\overline{\omega}(x)} \left( | \overline{\omega}(e^{\frac{k}{\mathrm{w}}}) f(e^{\frac{k}{\mathrm{w}}}) -\overline{\omega}(x) f(x)| \right) \right]\\
    \leq & \frac{1}{\eta(x)}(I_{1} + I_{2})
\end{align*}
Let us consider $I_{1}$. By a straight forward computation , we obtain
\begin{align*}
    I_{1} \leq & \|f\|_{\overline{\omega}} \bigvee_{k \in \mathbf{Z}}|\Chi(e^{-k}x^{\mathrm{w}})| \left|\left(\frac{k}{\mathrm{w}}\right)^{2} - \log^{2}(x)\right|\\
    \leq & \|f\|_{\overline{\omega}} \bigvee_{k \in \mathbf{Z}}|\Chi(e^{-k}x^{\mathrm{w}})| \left(\left|\frac{k}{\mathrm{w}} - \log(x)\right|^{2} + 2|\log(x)| \left|\frac{k}{\mathrm{w}} - \log(x)\right| \right)\\
    \leq & \|f\|_{\overline{\omega}} \left( \bigvee_{k \in \mathbf{Z}}|\Chi(e^{-k}x^{\mathrm{w}})| \frac{|k-\mathrm{w}\log(x)|^{2}}{\mathrm{w}^{2}}   +   \bigvee_{k \in \mathbf{Z}}|\Chi(e^{-k}x^{\mathrm{w}})| \frac{2|\log(x)|}{\mathrm{w}} |k-\mathrm{w}\log(x)|\right)\\
    \leq & \|f\|_{\overline{\omega}} \left( \frac{m_{2}(\Chi)}{\mathrm{w}^{2}} +\frac{2|\log(x)|}{\mathrm{w}} m_{1}(\Chi) \right)
\end{align*}
Let us consider $I_{2}$.Let $x \in \mathrm{R_{+}} \And{\epsilon > 0}$. Since f is log-Continuous at x then $\overline{\omega}f$ is also log-Continuous at x.Hence $\exists \delta > 0$ such that \[\left| \overline{\omega}(e^{\frac{k}{\mathrm{w}}}) f(e^{\frac{k}{\mathrm{w}}}) -\overline{\omega}(x) f(x)\right| < \epsilon \; \text{whenever} \left|\frac{k}{\mathrm{w}} - \log(x)\right| \leq \delta \]
Then we can write 
\begin{align*}
    I_{2} = & \bigvee_{k\in\mathbf{Z}} |\Chi(e^{-k}x^{\mathrm{w}})| \left[\frac{1}{\overline{\omega}(x)} \left( | \overline{\omega}(e^{\frac{k}{\mathrm{w}}}) f(e^{\frac{k}{\mathrm{w}}}) -\overline{\omega}(x) f(x)| \right) \right]\\
    = & \bigvee_{\substack{k\in\mathbf{Z} \\ |k-\mathrm{w}\log(x)|\leq \delta}} | \Chi(e^{-k}x^{\mathrm{w}})| \left[\frac{1}{\overline{\omega}(x)} \left( | \overline{\omega}(e^{\frac{k}{\mathrm{w}}}) f(e^{\frac{k}{\mathrm{w}}}) -\overline{\omega}(x) f(x)| \right) \right] \\ & +\bigvee_{\substack{k\in\mathbf{Z} \\ |k-\mathrm{w}\log(x)|> \delta}} | \Chi(e^{-k}x^{\mathrm{w}})| \left[\frac{1}{\overline{\omega}(x)} \left( | \overline{\omega}(e^{\frac{k}{\mathrm{w}}}) f(e^{\frac{k}{\mathrm{w}}}) -\overline{\omega}(x) f(x)| \right) \right]\\
    = & I_{2.1}+I_{2.2}
\end{align*}
Hence we can write 
\begin{align*}
    I_{2.1} = & \bigvee_{\substack{k\in\mathbf{Z} \\ |k-\mathrm{w}\log(x)|\leq \delta}}| \chi(e^{-k}x^{\mathrm{w}})| \left[\frac{1}{\overline{\omega}(x)} \left( | \overline{\omega}(e^{\frac{k}{\mathrm{w}}}) f(e^{\frac{k}{\mathrm{w}}}) -\overline{\omega}(x) f(x)| \right) \right] \\
    \leq & \frac{\epsilon}{\overline{\omega}(x)} \bigvee_{k\in\mathbf{Z}} |\Chi(e^{-k}x^{\mathrm{w}})| \\
    \leq & \frac{\epsilon m_{0}(\Chi)}{\overline{\omega}(x)}
\end{align*}
on the other hand by lemma 
\begin{align*}
    I_{2.2} = & \bigvee_{\substack{k\in\mathbf{Z} \\ |k-\mathrm{w}\log(x)|> \delta}} | \Chi(e^{-k}x^{\mathrm{w}})| \left[\frac{1}{\overline{\omega}(x)} \left( | \overline{\omega}(e^{\frac{k}{\mathrm{w}}}) f(e^{\frac{k}{\mathrm{w}}}) -\overline{\omega}(x) f(x)| \right) \right]\\
    = & \frac{2\|f\|_{\overline{\omega}}}{\overline{\omega}(x)} \bigvee_{\substack{k\in\mathbf{Z} \\ |k-\mathrm{w}\log(x)|> \delta}} |\Chi(e^{-k}x^{\mathrm{w}})|\\
    \leq & \frac{2 \epsilon \|f\|_{\overline{\omega}}}{\overline{\omega}(x)}
\end{align*}
Therefore, combining the estimates $I_{1}, I_{2.1} \And{I_{2.2}} $, we have
\begin{align*}
    |MG^{\chi}_{\mathrm{w}}(f,x) - f(x)| \leq  & MG^{\chi}_{\mathrm{w}}(|f-f_{x}|,x)\\
    \leq & \frac{1}{\eta(x)}\left( \|f\|_{\overline{\omega}} \left( \frac{m_{2}(\Chi)}{\mathrm{w}^{2}} +\frac{2|\log(x)|}{\mathrm{w}} m_{1}(\Chi) \right)+\frac{\epsilon m_{0}(\Chi)}{\overline{\omega}(x)} + \frac{2 \epsilon \|f\|_{\overline{\omega}}}{\overline{\omega}(x)} \right) \numberthis \label{eqn}
\end{align*}
Taking limit both sides as $\mathrm{w} \to \infty$ we have 
\[\lim_{\mathrm{w} \to \infty } MG^{\chi}_{\mathrm{w}}(f,x) = f(x)\]
For $f \in \mathcal{UB}^{\overline{\omega}}_{+}(\mathrm{R_{+}}) $ and from the inequality (3.1) , we have 
\begin{align*}
    |MG^{\chi}_{\mathrm{w}}(f,x) - f(x)| \leq & \frac{1}{\eta(x)}\left( \|f\|_{\overline{\omega}} \left( \frac{m_{2}(\Chi)}{\mathrm{w}^{2}} +\frac{2|\log(x)|}{\mathrm{w}} m_{1}(\Chi) \right)+\frac{\epsilon m_{0}(\Chi)}{\overline{\omega}(x)} + \frac{2 \epsilon \|f\|_{\overline{\omega}}}{\overline{\omega}(x)} \right) \\
    \overline{\omega}(x)|MG^{\chi}_{\mathrm{w}}(f,x) - f(x)| \leq & \frac{1}{\eta(x)}\left( \overline{\omega}(x) \|f\|_{\overline{\omega}} \left( \frac{m_{2}(\Chi)}{\mathrm{w}^{2}} +\frac{2|\log(x)|}{\mathrm{w}} m_{1}(\Chi) \right)+\epsilon m_{0}(\Chi) + 2 \epsilon \|f\|_{\overline{\omega}} \right)    
\end{align*}
By taking supremum over $x \in \mathrm{R_{+}} \And{\mathrm{w} \to \infty}$ , we have
\[\lim_{\mathrm{w} \to \infty } \|MG^{\chi}_{\mathrm{w}}(f,x) - f(x)\|_{\overline{\omega}} = 0 \]
\end{proof}

\begin{thm}
    Let $\chi$ be a kernel satisfying the assumptions $(\Chi_{1}) \And{(\chi_{2})} $ for $\mu = 5$.Then $f \in \mathcal{UB}^{\overline{\omega}}_{+}(\mathrm{R_{+}}) $ 
    \[\| MG^{\chi}_{\mathrm{w}}(f,x) - f(x)\|_{w} \leq \frac{64 \Omega(f,\frac{1}{w})}{\eta(x)} [M_{0}(\chi) + M_5(\chi)] \]
\end{thm}
\begin{proof}
    Let $x \in \mathrm{R_{+}}$ be fixed. From the definition of the sampling operators $MG^{\chi}_{\mathrm{w}} $  and using the definition of logarithmic modulus of continuity we can write 
    \[ |MG^{\chi}_{\mathrm{w}} (f;x)- f(x)| \leq \frac{1}{\eta(x)} \bigvee_{k\in\mathbf{Z}} | \chi(e^{-k}x^{\mathrm{w}})| |f(e^{\frac{k}{\mathrm{w}}}- f(x)| \]
    Since  for all $f \in \mathcal{UB}^{\overline{\omega}}_{+}(\mathrm{R_{+}})$
    \[ |(f(h)- f(x)|\leq 16 (1+\delta^{2})^{2} (1+\log^{2}x) \left(1+ \frac{|\log h-\log x|^{5}}{\delta^{5}}\right) \Omega(f,\delta)\]
    \[\Rightarrow |(f(e^{\frac{k}{w}})- f(x)|\leq 16 (1+\delta^{2})^{2} (1+\log^{2}x) \left(1+ \frac{|k-w\log x|^{5}}{w^{5}\delta^{5}}\right) \Omega(f,\delta) \]
    So  for any positive $\delta \leq 1$ we have
    \begin{align*}
	|MG^{\chi}_{\mathrm{w}} (f;x)- f(x)| \leq & \frac{1}{\eta(x)}\bigvee_{k\in \mathbf{Z}} \chi(e^{-k}x^{w})| 64 (1+\log^{2}x) \left(1+ \frac{|k-w\log x|^{5}}{w^{5}\delta^{5}}\right) \Omega(f,\delta)\\
	\leq & \frac{64(1+\log^{2}x)\Omega(f,\delta)}{\eta(x)} \bigvee_{k\in \mathbf{Z}} | \chi(e^{-k}x^{w})|\left(1+ \frac{|k-w\log x|^{5}}{w^{5}\delta^{5}}\right)\\
	\leq & \frac{64(1+\log^{2}x) \Omega(f,\delta)}{\eta(x)}[M_{0}(\chi) + \frac{1}{w^5 \delta^{5}} M_5{\chi}] 
    \end{align*}
    Finally choosing $\delta = \frac{1}{w}$ $w \geq 1$  we have 
    \[|MG^{\chi}_{\mathrm{w}} (f;x)- f(x)| \leq \frac{64(1+\log^{2}x) \Omega(f,\delta)}{\eta(x)}[M_{0}(\chi) + \frac{1}{w^5 \delta^{5}} M_5{\chi}]\]
    which is desired.
\end{proof}
\\
Let $j \in \mathbf{N} $, then  the algebraic moment of order j of a kernel defined by 
\[ M_{j}(\chi,u) =  \bigvee_{k \in \mathbf{Z}} \Chi(e^{-k}x) (k-\log x)^{j} \; \; \text{$\forall x \in \mathrm{R}_{+}$} \]. 
\par 
In order to present Voronovskaja theorem in quantitative form we need a further assumption on kernel function $\Chi$,i.e there exists $r\in \mathbf{N}$ such that for every $j \in \mathbf{N}_{0}$, $j\leq r$ there holds:\\
$(\Chi_{3} ): M_{j}(\Chi,x) = M_{j}(\Chi)\in \mathrm{R}_{+} $\;  is independent of x.\\
Let us consider the Mellin-Taylor formula given in ... For any $f\in \mathcal{UB}^{\overline{\omega}}_{+}(\mathrm{R_{+}})$ belonging to $\mathcal{C}^{r}(\mathrm{R}_{+})$ locally at the point $x\in \mathrm{R}_{+}$ , the Mellin formula with the Mellin derivatives is defined by 

\[ f(t) = \sum_{t=0}^{r} \frac{\theta^{t}f(x)}{t!} (\log u -\log x)^{t} + \mathcal{R}_r(f;u,x)\]\\
Where \[ \mathcal{R}_{r}(f;u,x) = \frac{\theta^{r}f(\eta) -\theta^{r}f(x) }{r!}(\log u-\log x)^{r} \]\\
is the Lagrange remainder in Mellin-Taylor's formula at the point $x\in \mathrm{R}_{+}$ and $\vartheta $ is a number lying  between u and x.According the inequality 
\[|(f(h)- f(x)|\leq 16 (1+\delta^{2})^{2} (1+\log^{2}x) (1+ \frac{|\log h-\log x|^{5}}{\delta^{5}}) \Omega(f,\delta)\], with the similar method presented in [\cite{bardaro9}] ,we can easily have the estimate 
\[\left| \mathcal{R}_r(f;u,x) \right| \leq  \frac{64}{r!} (1+\log^{2}x) \Omega(\theta^{r}f,\delta) \left( |\log u-\log x|^{r} + \frac{|\log u-\log x|^{r+5}}{\delta^{5}}\right) \]
\begin{thm}
    Let $\chi$ be the kernel satisfying the assumptions  $(\chi_{1})$  $(\chi_{2})$ and $(\chi_{3})$ for $\mu = r+5 $  $ r \in \mathbf{N}$ and in addition $(\Chi_{3})$  is satisfied   for every $ x \in \mathrm{R}_{+}.$If $f^{r} \in \mathcal{UB}^{\overline{\omega}}_{+}(\mathrm{R_{+}})$ for $ f \in \mathcal{C}^{r}(\mathrm{R}_{+}) $ then we have for every $x \in \mathrm{R}_{+}$ that
    \begin{align*}
	\left | \mathrm{w}^{r} \left[MG^{\chi}_{\mathrm{w}} (f;x) - \frac{1}{M_{0}(\Chi)}\sum_{t=0}^{r } \frac{\theta^{t}f(x)}{t! \mathrm{w}^{t}} M_{t}(\chi)\right]\right| \leq & \frac{64}{r!M_{0}(\Chi)} (1+\log^{2}x) \Omega(\theta^{r}f,\mathrm{w}^{-1}) ( m_{r}(\chi) + m_{r+5}(\Chi)) \\
    \end{align*}
\end{thm}


\begin{proof}
    Since $f^{r} \in \mathcal{UB}^{\overline{\omega}}_{+}(\mathrm{R_{+}})$ , using Mellin-Taylor expansion at a point $x \in \mathrm{R}_{+}$ and by definition of $MG^{\chi}_{\mathrm{w}}(f)$ we can write
    \begin{align*}
	  MG^{\chi}_{\mathrm{w}} (f;x) = &\frac{1}{M_{0}(\Chi)}\bigvee_{k\in \mathbf{Z}} \Chi(e^{-k}x^{\mathrm{w}}) \left[ \sum_{t=0}^{r } \frac{\theta^{t}f(x)}{t! \mathrm{w}^{t}} (k-\mathrm{w}\log x)^{t} \right] \\ & + \frac{1}{M_{0}(\Chi)} \bigvee_{k\in \mathbf{Z}} \chi(e^{-k}x^{\mathrm{w}}) \mathcal{R}_{r}(f;\frac{k}{\mathrm{w}},x) \\
	  = & I_{1}+ I_{2} 
    \end{align*}
    
    Let us first consider $I_{1}$
    \begin{equation*}
	\begin{split}
		I_{1} = & \frac{1}{M_{0}(\Chi)}\bigvee_{k\in \mathbf{Z}} \chi(e^{-k}x^{\mathrm{w}}) \left[ \sum_{t=0}^{r } \frac{\theta^{t}f(x)}{t! \mathrm{w}^{k}} (k-\mathrm{w}\log x)^{t} \right]\\
		= & \frac{1}{M_{0}(\Chi)}\sum_{t=0}^{r } \frac{\theta^{t}f(x)}{t! \mathrm{w}^{t}}\bigvee_{k\in \mathbf{Z}} \chi(e^{-k}x^{\mathrm{w}})(k-\mathrm{w}\log x)^{t}\\
		= & \frac{1}{M_{0}(\Chi)}\sum_{t=0}^{r } \frac{\theta^{t}f(x)}{t! \mathrm{w}^{t}} M_{t}(\Chi,x)
	\end{split}
    \end{equation*}
    On the other-hand concerning $I_{2}$ by the inequality
    \[\left| \mathcal{R}_r(f;\frac{k}{\mathrm{w}},x) \right| \leq \frac{64}{r! \mathrm{w}^{r}} (1+log^{2}x) \Omega(\theta^{r}f,\delta) \left( |k-\mathrm{w}\log x|^{r} + \frac{|k-\mathrm{w}\log x|^{r+5}}{(\mathrm{w}\delta)^{5}}\right) \]we have 
    \begin{equation*}
        \begin{split}
            |I_{2}| \leq & \frac{1}{M_{0}(\Chi)}\bigvee_{k\in \mathbf{Z}} |\chi(e^{-k}x^{\mathrm{w}})| |\mathcal{R}_{r}(f;\frac{k}{\mathrm{w}},x)|\\
		\leq & \frac{1}{M_{0}(\Chi)}\bigvee_{k\in \mathbf{Z}} |\chi(e^{-k}x^{\mathrm{w}})| \frac{64}{r! \mathrm{w}^{r}} (1+\log^{2}x) \Omega(\theta^{r}f,\delta) \left( |k-\mathrm{w}\log x|^{r} + \frac{|k-\mathrm{w}\log x|^{r+5}}{(\mathrm{w}\delta)^{5}}\right) \\
		\leq & \frac{1}{M_{0}(\Chi)}\frac{64}{r! \mathrm{w}^{r}} (1+\log^{2}x) \Omega(\theta^{r}f,\delta) \left(\bigvee_{k\in \mathbf{Z}} |\chi(e^{-k}x^{\mathrm{w}})| |k-\mathrm{w}\log x|^{r} + \bigvee_{k\in \mathbf{Z}} |\chi(e^{-k}x^{\mathrm{w}})|\frac{|k-\mathrm{w}\log x|^{r+5}}{(\mathrm{w}\delta)^{5}}\right)\\
		\leq & \frac{1}{M_{0}(\Chi)}\frac{64}{r! \mathrm{w}^{r}} (1+\log^{2}x) \Omega(\theta^{r}f,\delta) ( m_{r}(\chi) + m_{r+5}(\chi))
        \end{split} 
    \end{equation*}
    Hence choosing $ \delta = \mathrm{w}^{-1}$ $\mathrm{w}\geq 1$ and applying condition of $(\Chi_{3})$, we have
    $$ \left | \mathrm{w}^{r} \left[MG^{\chi}_{\mathrm{w}} (f;x) - \frac{1}{M_{0}(\chi)}\sum_{k=0}^{r } \frac{\theta^{k}f(x)}{k! \mathrm{w}^{k}} M_{k}(\chi)\right]\right| \leq  \frac{1}{M_{0}(\Chi)}\frac{64}{r!} (1+\log^{2}x) \Omega(\theta^{r}f,\mathrm{w}^{-1}) ( m_{r}(\chi) + m_{r+5}(\chi)).$$
\end{proof}
	



	\subsection*{Conflicts of Interest}
	There is no conflict of interest regarding the publication of this article

\end{document}